\documentclass[11pt, reqno]{amsart}
\usepackage{amsmath}
\usepackage{amsthm}
\usepackage{amssymb}
\usepackage{amscd}
\usepackage{amsbsy}
\usepackage{epsfig}
\usepackage{pgfplots}
\usepackage{graphicx}
\usepackage{psfrag}
\usepackage{bbm}
\usepackage{hyperref}

\usepackage[UKenglish]{babel}
\usepackage[UKenglish]{isodate}
\usepackage{cleveref}
\usepackage{enumerate}

\theoremstyle{plain}
\newtheorem{theorem}{Theorem}[section]

\newtheorem{remark}[theorem]{Remark}
\newtheorem{proposition}[theorem]{Proposition}
\newtheorem{lemma}[theorem]{Lemma}

\theoremstyle{definition}
\newtheorem{definition}[theorem]{Definition}
\newtheorem{notation}[theorem]{Notation}

\newtheorem*{claim*}{Claim}

\newtheorem*{lemma*}{Lemma}

\theoremstyle{definition}

\def\R{\ensuremath{\mathbb R}}
\def\N{\ensuremath{\mathbb N}}

\def\Z{\ensuremath{\mathbb Z}}
\def\T{\ensuremath{\mathbb T}}

\numberwithin{equation}{section}

\begin{document}

\author[S.~Shaabanian]{Saeed Shaabanian}
\thanks{I would like to thank my supervisor, Prof.\ Mike Todd, for their guidance, valuable insights, and support throughout this research.
I kindly thank Prof.\ Nicolai Haydn for useful suggestions to improve the previous version of the paper. }
\address{Mathematical Institute\\
University of St Andrews\\
North Haugh\\
St Andrews\\
KY16 9SS\\
Scotland} 

\email{ss507@st-andrews.ac.uk}

\title{Hitting statistics for $\phi$-mixing dynamical systems}
\date{\today}
\maketitle

\begin{abstract}
Hitting rate and escape rate are two examples of recurrence laws for a dynamical system, and a general limit connects them. We show that for both Gibbs-Markov systems or any systems with the $\phi$-mixing measure, for a sequence of nested sets whose intersection is a measure zero set, this general limit equals one in the absence of short returns and less than one otherwise, which is given by an explicit formula called extremal index. One of the applications of this result is to dynamical systems on Riemannian manifolds such as hyperbolic maps and expanding maps, and it can be applied to any system with a suitable Young tower.
\end{abstract}
\section{Introduction}\label{Sec:Int}
An open dynamical system is a system where there is a hole. This means that the points of the system that enter this hole are removed from the system. Recently, there has been a surge of interest in open dynamical systems, leading to significant advancements in our understanding of such systems.

In open dynamical systems, we have two laws of recurrence, the return of a point of the system to the sets that shrink towards that point, which is called the hitting time of that point, and the other is the entry of a point of the system into a fixed set (hole) which is called escape rate.
If we consider the hitting time from the point of view of a random variable, it often has an exponential distribution, see \eqref{a} below. If we consider localized escape rates which one obtains when the size of the hole shrinks to zero around the point, then there is a relationship between these two, see \eqref{Equation: escape rate of L} below. As a result in \cite{BDTodd18}, a general formula has been obtained in which, by choosing the appropriate parameters, we have each of them. This general formula is obtained in the space of open systems with interval maps and conformal measures.

On the other hand, the escape rate and the hitting time have also been obtained if our system has a $\phi$-mixing measure, but in this setting, there is no result of the general formula as in \cite{BDTodd18}, that is our main goal in this research.

The results obtained in this paper show that the general formula has almost the same behaviour in $\phi$-mixing systems with slow mixing or in Gibbs-Markov systems, which are a fast mixing systems, as it does in \cite{BDTodd18}. In fact, if the system has good behaviour and is exponentially $\phi$-mixing, the general formula holds in more cases. However, if the system is polynomially $\phi$-mixing, the general formula applies to fewer cases. The number of cases where the general formula holds depends on the rate of mixing and the rate of decay of the measure of cylinders. 

Note that the general formula obtained here applies to zero measure sets, which is more general in comparison to the setting of \cite{BDTodd18}, where the general formula for singletons is proved. At the end of the paper we provide examples where our theory applies to zero measure sets such as Cantor sets for an interval map (see Section \ref{Section: Cantor set}), and the line segment for the Arnold cat map (see Section \ref{Example: Cat map}).

\subsection{Organization of the paper}
We will first review the definitions and previous results in Section 1, then in Section 2, we present the main theorem (Theorem \ref{mainTH}). In Section 3, we address the general formula for Gibbs-Markov systems. In Section 4, we demonstrate its application which is use of open balls to derive the general formula. Finally, the paper concludes in Section 5 by providing some examples.

\subsection{Hitting times, escape rates, and their relation}\label{Section: Mike Theorem}Suppose $(X,\mathcal{B},T,\mu)$ as an ergodic probability measure preserving \mbox{dynamical} system i.e. $X$ is a space, $\mathcal{B}$ is a $\sigma$-algebra on $X$, $T:X\to X$ is a measurable map and $\mu$ is an ergodic probability measure on $(X,\mathcal{B})$ such that \mbox{$\mu(T^{-1}B)=\mu (B)$} for each $B\in \mathcal{B}$.

\begin{definition}The \textit{hitting time} of $z$ to the set $U$ is the first time that the orbit of $z$ (i.e. $T^nz$ for $n=1,2,3,... $) hits $U$ i.e.
    $$\tau_U (z):=\inf \{j \ge 1 : T^j z \in U\}.$$

\end{definition}

Poincar\'e
 shows for each $z \in U$, $T^nz$ belongs to $U$ infinitely often. After that, Kac in~\cite{Kac47} calculates the mean behaviour of hitting time.
 
Suppose that our system is an open system i.e., a hole is placed in the space and when points hit that hole, they are removed from the dynamics. We denote the holes by $\{U_r\}_r$ that shrink to a point $z$ as $r \to 0$. Let $s\in \mathbb{R}^+$. A natural question that can arise here is, what is $\mu(\tau_{U_r}(z)>s)$?
 
 Hitting time statistics started with Pitskel~\cite{Pitskel91}.  Pitskel shows that
 \begin{equation}
    \left|\mu\left(\mu(U_r) . \tau_{U_r}(z) >s\right) - e^{-s}\right| \to 0 \quad as \quad  r \to 0.\label{a}
\end{equation}

Pitskel's result was generated by Hirata~\cite{Hirata93} for hyperbolic dynamical systems. Then, Collet and Galves~\cite{CGalves95} developed this result in the context of piecewise-expanding maps. It gradually became clear that \eqref{a} holds in a wide range of dynamical systems, see~\cite[Chapter5]{Mike'sbook16}.

\noindent We can rewrite \eqref{a} as $\lim_{r \to 0} \log \mu(\mu(U_r) \tau_{U_r} >s)=-s$. So,
\begin{equation}
    \lim_{r \to 0} -\frac{1}{s} \log \mu \left(\tau_{U_r} > \frac{s}{\mu(U_r)}\right)=1.\label{b}
\end{equation}

\noindent Bruin, Demers and Todd (\cite{BDTodd18}) formulated a generalised limit of \eqref{b}. This limit can be seen as the special limiting case of the expression
    $$\frac{1}{\mu(U_r)} \frac{-1}{t} \log \mu \left(\tau_{U_r} >t\right),$$
by setting $t=\frac{s}{\mu(U_r)}$ . We can see the general behaviour of the system around point $z$, by setting $t=s\mu(U_r)^{-\alpha}$ for some $\alpha, s \in (0,\infty)$, and let,
    $$ L_{\alpha,s} (z):=\lim_{r \to 0} \frac{-1}{s\mu (U_r)^{1-\alpha} } \log(\mu(\tau_{U_r} > s\mu (U_r)^{-\alpha})).$$

So when $\alpha =1$, we have \eqref{b} again, when $\alpha = \infty$,  we can consider it as coinciding with the derivative of the escape rate, where $t \to \infty$ as $r$ is held fixed:
    \begin{equation}\label{Equation: escape rate of L}       
    \lim_{r \to 0} 
 \frac{1}{\mu (U_r)} \lim_{t \to \infty} \frac{-1}{t} \log (\mu (\tau_{U_r} >t)),\end{equation}
 
and while $\alpha=0$ we can consider it as the reversed order of limits, where $r \to 0$ as $t$ is held fixed:
    $$\lim_{t \to \infty}\lim_{r \to 0} \frac{-1}{t \mu (U_r)}\log (\tau_{U_r} >t).$$
The main result in~\cite[Theorem 2.1]{BDTodd18} expresses a limit for $L_{\alpha,s}$. We briefly state it below, but first need some description of the setting.

Let $f: I \to I$ be a piecewise continuous map of the unit interval and assume that $I$ has countably intervals of monotonicity. Also assume that there exists a countable collection of maximal intervals  $Z=\{Z_i\} \subset I$ such that $f$ is continuous and monotone on each $Z_i$ and $\forall  i\ne j$, \mbox{int$(Z_i)\ \cap$ int$(Z_j)=\emptyset $}. Now let $D:=I\setminus \cup_i$ int$(Z_i)$.

\noindent Define a Borel probability measure map $m_{\varphi}$ such that $m_{\varphi}(D)=0 $, which is conformal with respect to a potential $\varphi:I \to \mathbb{R}$ (i.e. $\frac{dm_{\varphi}}{d(m_{\varphi}\circ f)}=e^{\varphi}$). Put $$I_{cont}:=\left\{z\in I:f^k \text{ is continuous at } z \text{ for all } k\in \mathbb{N}\right\},$$ and suppose that $\mu_{\varphi}(I_{cont})=1$, where $\mu_{\varphi}$ is the equilibrium state of $\varphi$ i.e. $\mathbf{P}(\varphi)=h_{\mu_{\varphi}}+\int \varphi d\mu_{\varphi}$ ($\mathbf{P}$ is the Pressure of $\varphi$ ). Let $S_n\varphi(z)= \sum_{i=1}^{n-1}\varphi \circ f^i$, then we have the following result.

\begin{theorem}\label{Theorem:MainBDT}
\cite[Theorem 2.1]{BDTodd18} Let $z\in I_{cont}$ and $(U_r)_r$ be a family of intervals such that $\lim_{r\to 0}diam (U_r)=0$ and $\bigcap_r U_r=\{z\}$. 

Then for any $s\in \mathbb{R}^+, \alpha \in [0,\infty]$, taking $ L_{\alpha,s} (z)$ with respect to $\mu_{\varphi}$, we have:
\begin{equation} 
 L_{\alpha,s} (z) = 
  \begin{cases} 
   1 & \text{if $z$ is not periodic} \\
   1-e^{S_p\varphi(z)} & \text{if  $z$ is $p$-periodic,} 
  \end{cases}\label{d}
\end{equation}

\noindent where $p$-periodic means that the prime period of $z$ is $p$, i.e. the smallest positive integer $p$ such that $f^p(z)=z$.
\end{theorem}

\begin{remark}
    In the theorem above, in the case $z$ is periodic, we assume $f^p$ is monotonic and the density is continuous at $z$. In addition, $f$ and $\varphi$ must satisfy four regularity conditions which are mainly related to the existence of an upper bound for $e^{S_p\varphi(z)}$. Also, $\{U_r\}_r$ must satisfy a large image condition. We do not state them for brevity, but all the mentioned conditions can be found in \cite[Page1268]{BDTodd18}.
\end{remark}

\subsection{Escape rates for $\phi$-mixing dynamical systems}A part of Theorem \ref{Theorem:MainBDT} was proven by Davis, Haydn, and Yang in a different setting ~\cite[Theorem A and Corollary C]{HYang20}. Namely, they proved the case where $\alpha=\infty$ of Theorem \ref{Theorem:MainBDT} in a symbolic dynamic as a system with a $\phi$-mixing measure, which we will explain further.

Let $(\Omega, T,\mu)$ be a probability measure preserving dynamical system. Assume that there exists $\mathcal{A}=\{A_1, A_2,...\}$ a measurable partition of $\Omega$. 
\begin{definition}\label{Def:cylinders}
    The partition $\mathcal{A}^n=\bigvee^{n-1}_{i=0} T^{-i}\mathcal{A} $ of $\Omega$ is \textit{$n$-th join} of $\mathcal{A}$. The elements of $\mathcal{A}^n$ called $n$-\textit{cylinders}.
\end{definition} 
Note that $\sigma (\mathcal{A}^n)$ is the $\sigma$ algebra generated by $\mathcal{A}^n$. Some key definitions follow.
\begin{definition}
    
The probability measure $\mu$ is \textit{left $\phi$-mixing} with respect to $\mathcal{A}$ if
\begin{center}\label{phidef}
    $|\mu(A \cap T^{-n-k}B)- \mu (A) \mu(B)| \le \phi(k) \mu (A),$
\end{center}
and is \textit{right $\phi$-mixing} with respect to $\mathcal{A}$ if
    $$|\mu(A \cap T^{-n-k}B)- \mu (A) \mu(B)| \le \phi(k) \mu (B),$$
    and $\psi$-mixing with respect to $\mathcal{A}$ if 
    $$|\mu(A \cap T^{-n-k}B)- \mu (A) \mu(B)| \le \phi(k) \mu(A)\mu (B),$$
    
\noindent for all $A \in \sigma (\mathcal{A}^n)$, $n\in \mathbb{N}$ and $B\in \sigma (\bigcup_j \mathcal{A}^j)$, where $\phi(k)$ is a decreasing function which converges to zero as $k \to \infty$.
\end{definition}

\begin{remark}\label{remark:phi=psi}
    Every $\psi$-mixing measure is left and right $\phi$-mixing by choosing $\psi(k)=\phi(k)$.
\end{remark}

\begin{remark}
    
Whenever we say $(X, T,\mu)$ is a $\phi$-mixing dynamical system, we mean that the measure is (left or/and right) $\phi$-mixing.
\end{remark}

\begin{remark}\label{Remark:Right-Left}
    
Typically, the concept of $\phi$-mixing in sources refers to what is defined as left $\phi$-mixing. Left and right $\phi$-mixing are asymmetric (because the term $\mu(A \cap T^{-n-k}B)$ is fixed in both definitions), but we can derive results for right $\phi$-mixing with slight modifications if we have results for left $\phi$-mixing.
\end{remark}

This property is one of the five strong mixing properties for probability measures (see \cite{Bradley05}). The $\phi$-mixing measure was first defined by Ibragimov in 1962 (see \cite{Ibra62} and \cite{IbraRoz78}). Many features of this property, along with examples, can be found in the book \cite[Chapter1]{Doukh95}. The history of the first study of statistical properties related to the hitting rate in mixing processes can be found in \cite{AbadiGa01}. However, Abadi was the first to study these properties in the space with a $\phi$-mixing measure (\cite{Abadi01}). He later found an upper bound for the hitting time law in $\phi$-mixing systems (\cite{Abadi04}). Five years later, in collaboration with Vergne, they proved \eqref{d} for the case $\alpha=1$ in $\phi$-mixing systems (\cite{AbadiVer09}).

We can define the escape rate to $U\subset\Omega$.
\begin{definition}
    Denote the \emph{escape rate} to $U$  by $ \rho (U)$ and define it as $$ \rho (U):= \lim_{t \to \infty} \frac{1}{t}|\log( \mu (\tau_U > t))|,$$if the limit exists.
\end{definition}
\noindent Now, we can find the localized escape rate. For this let $(U_n)_n$ be \textit{nested sequence} of sets i.e. $\forall n,\  U_{n+1}\subset U_n$ and $\lim_n \mu(U_n)=0$. Note that shrinking the $(U_n)_n$ to the point $z\in \Omega $ is a specific case of these settings. In fact, the result obtained in~\cite{HYang20} is more general.
Then we define the \textit{localized escape rate} for the measure zero set $\Lambda=\cap_n U_n$ as,
\begin{align*}  
 \rho(\Lambda):=\lim_{n \to \infty} \frac{\rho (U_n)}{\mu (U_n)}=\lim_{ n \to \infty} \frac{1}{\mu(U_n)} \lim_{t \to \infty} \frac{1}{t}|\log (\mu (\tau_{U_n} > t))|,
\end{align*}
whenever the limit exists. 

We consider, for each $U_n$, a union of $\kappa_n$-cylinders, based on a non-decreasing sequence of integers $\left\{\kappa_n\right\}_n$. For any set $U_n$, we can approximate it in $\sigma (\mathcal{A}^j)$ using the following definition.
\begin{definition}
For each $n$ and $j\ge 1$, set $C_j(U_n)=\left\{B\in \mathcal{A}^j:B\cap U_n\ne \emptyset\right\}$, then we define $$ U_n^j:= \bigcup_{B\in C_j(U_n)} B,$$ as the \textit{outer $j$-cylinder approximation of $U_n$}.
\end{definition}

That is the union of those cylinders in $\mathcal{A}^j$ that have an intersection with $U_n$.

\begin{remark}
\begin{enumerate}

    \item $\forall j,\  U_{n+1}^j\subset U_{n}^j$ and $U_{n}\subset U_{n}^j$,
    \item If $j\ge \kappa_n,\ U_{n}=U_{n}^j $ .

\end{enumerate}
\end{remark}

The following sequence of cylinders is also suitable for our purpose.

\begin{definition}\label{adapt}
   Consider nested sequence $\left\{U_n\in \mathcal{A}^{\kappa_n}\right\}$ for $n=1,2,...$, if this sequence has the following two properties, we call it \textit{a good neighbourhood system}:
\begin{enumerate}[(N1)]
    \item $\kappa_n\nearrow\infty$ and $\kappa_n\mu(U_n)^c\to 0$ for some $c\in(0,1)$; \label{Def:adapted1}
    \item there exists $C>0$ and $m'>1$ so that $\mu(U_n^j)\le \mu(U_n)+Cj^{-m'}$ for all $j<\kappa_n$.\label{Def:adapted2}
\end{enumerate}
\end{definition}

\begin{definition}
    Let $U\subset \Omega$ be a nonzero measure set. Considering the definition of hitting time, we can also derive its \textit{higher order} as follows: $$ \tau^0_U:=0,$$ $$\tau^1_U:=\tau_U,$$$$\tau^j_U(x):=\tau^{j-1}_U(x)+\tau_U\left(T^{\tau_U^{j-1}}(x)\right).$$
\end{definition}
Now suppose that $(U_n)_n$ is a nested sequence in $\Omega$ and $K\in\N$, then $$\hat\theta_l(K,U_n):=\mu_{U_n}\left(\tau_{U_n}^{l-1}\le K\right),$$
is the conditional probability of having at least $(l-1)$ returns to $U_n$ before time $K$.
Assume that $$\hat\theta_l(K):=\lim_{n\to\infty}\hat\theta_l(K,U_n),$$ exists for sufficiently large $K$ and $l\ge1$. Then the following limit also exists (due to $\hat\theta_l(K)\le \hat\theta_l(K')\le 1$ for $K\le K')$, \begin{equation}\label{Def:theta_l}
    \hat\theta_l:=\lim_{K\to\infty}\hat\theta_l(K).
\end{equation}
\begin{remark}
    $\hat\theta_1=1$, $\hat\theta_2=\lim_{K\to\infty}\lim_{n\to\infty}\mu_{U_n}\left(\tau_{U_n}\le K\right).$
\end{remark}
\begin{definition}
    Let $(U_n)_n$ be a nested sequence. The following limit is called \textit{extremal index},$$\theta_1:=\lim_{K\to\infty}\lim_{n\to\infty}\mu_{U_n}\left(\tau_{U_n}> K\right)=1-\hat\theta_2.$$
\end{definition}
\begin{remark}\label{Remark:Thetha}
    In many sources, such as \cite{FFTodd13} and \cite{Mike'sbook16}, the extremal index is given as $\theta=\lim_{n\to\infty}\frac{\mu(\tau_{U_n>K_n})}{\mu(U_n)}$, where $(K_n)_n\nearrow\infty$ is a sequence of integers. We will show that, for the results obtained in this paper, $\theta=\theta_1$.
\end{remark}

    If $(U_n)_n$ is a sequence in the absence of short returns, one can obtain $\hat\theta_2=0$ and consequently $\theta_1=1$. Using the definition below, this sequence is formulated.
    \begin{definition}
        Let $U\subset\Omega$ be a non zero measure set. Then the \textit{period of $U$} is $$\pi(U)=\min\{k>0: T^{-k}U\cap U\ne \emptyset\},$$ and the \textit{essential period of $U$} is $$\pi_{ess}(U)=\min\{k>0: \mu(T^{-k}U\cap U)>0\}.$$
    \end{definition}
\begin{remark}
    \begin{enumerate}
        \item $\pi(U)\le \pi_{ess}(U);$
        \item $\pi(U)=\pi_{ess}(U)$ if $T$ is continuous, $U$ is open and $\mu$ has full support. 
    \end{enumerate}
\end{remark}

Now, if $\pi_{ess}(U_n)\to \infty$ as $n\to\infty$, it means that the sequence $(U_n)_n$ does not have short returns.

\begin{remark}
    There are equivalent conditions for $\pi_{ess}(U_n)\to \infty$ in \cite[Corollary D]{HYang20}.
\end{remark}
 Davis, Haydn, and Yang proved the following theorem.

\begin{theorem}\label{DHY_TH}
      \cite[Theorem A and Corollary C]{HYang20} Let $(\Omega, T,\mu)$ be a probability measure preserving dynamical system such that $\mu$ is right $\phi$-mixing with $\phi(k)\le Ck^{-m}$ for some $C>0$ and $m>1$, and $(U_n)_n$ is a good neighbourhood system such that in the case $\pi_{ess}(U_n)<\infty$, $(\hat\theta_l)_l$ exists, and satisfies $\sum_l l \hat\theta_l<\infty$. Then $\theta_1$ and the localized escape rate at $\Lambda$ exists and satisfies,
\[  \rho (\Lambda) = 
  \begin{cases} 
   1 & \text{if $\pi_{ess}(U_n)\to \infty$} \\
   \theta_1 & \text{if $\pi_{ess}(U_n)<\infty$.} 
  \end{cases}
\]
\end{theorem}

\begin{remark}\label{Remark:DHY's Theor remarks}
    Theorem \ref{DHY_TH} still holds if:
    \begin{enumerate}
        \item $\mu$ is left $\phi$-mixing instead of right $\phi$-mixing \cite[Remark 2.2]{HYang20};
        \item The system is Gibbs-Markov instead of $\phi$-mixing system \cite[Theorem B]{HYang20}.
    \end{enumerate}
\end{remark}
\begin{remark}
    In Theorem \ref{DHY_TH}, $T$ is considered non-invertible. Although the theorem holds for an invertible map as well, it requires minor adjustments \cite[Remark 2.3]{HYang20}.
\end{remark}

 Comparing Theorem \ref{Theorem:MainBDT} with the Theorem \ref{DHY_TH} gives us the idea that it is possible to show \eqref{d} in the $\phi$-mixing setting as well, but as can be seen, Theorem \ref{Theorem:MainBDT} only considers single points, while here we can handle general zero-measure sets. Therefore, we aim to find $L_{\alpha,s} (\Lambda)$ in this more general setting.

\section{Recurrence rate for an polynomially $\phi$-mixing system}

In this section, we compute $L_{\alpha,s} (\Lambda)$ in different cases of $\alpha\in [0,\infty]$ for a $\phi$-mixing system. First, we present the main theorem, and then we prove it.

\subsection{Main Theorem}

In this section, we present the main theorem of this paper and proceed with its proof.

\begin{theorem}\label{mainTH}

      Let $(\Omega, T,\mu)$ be a $\phi$-mixing dynamical system with a measurable partition $\mathcal{A}$ of $\Omega$ such that $\phi$ decays at least polynomially with a power $m$ i.e. $\phi(k)\le Ck^{-m}$ for some $C>0$ and $m>1$. Assume that $T:\Omega\circlearrowleft$ preserves $\mu$. Suppose $(U_n)_{n\in \N}$ is a good neighbourhood system such that $\bigcap_nU_n=\{\Lambda\}$ where $\Lambda$ is a measure zero set. If $\pi_{ess}(U_n)<\infty$, then assume that $(\hat\theta_l)_l$ exists, and satisfies $\sum_l l \hat\theta_l<\infty$. Let $\alpha$ be a number in $[0,\infty)$ such that for $\alpha\in(0,1)$, we have $\alpha>\max\{\frac{1}{m+1},c\}$, where $c$ as defined in (N\ref{Def:adapted1}). Then for any $s\in\R^+$, $\theta_1$ exists and,
\[  L_{\alpha,s} (\Lambda) = 
   \begin{cases} 
   1 & \text{if $\pi_{ess}(U_n)\to \infty$} \\
   \theta_1 & \text{if $\pi_{ess}(U_n)<\infty$.} 
  \end{cases}
\]

\end{theorem}

\begin{remark}
    If $\pi_{ess}(U_n)\to \infty$, by Lemma \ref{Lemma:DHY_EI=1}, $ L_{\alpha,s} (\Lambda) =1$. So in this case the assumption of existence of $(\hat\theta_l)_l$, and the condition $\sum_l l \hat\theta_l<\infty$ are not required.
\end{remark}

\begin{remark}
    As stated in Section \ref{Sec:Int}, the case $\alpha=\infty$ equals the localized escape rate. This case is proved in~\cite[Theorem A and Corollary C]{HYang20}.
\end{remark}

\begin{remark}
    If $\alpha=0$, Theorem \ref{mainTH} holds. We prove it in Section \ref{Section: alpha=0}. Note that in this case, only the following conditions are required, $(\hat\theta_l)_l$ exist, and $\sum_l  \hat\theta_l<\infty$ where $\pi_{ess}(U_n)<\infty$.
\end{remark}

\begin{remark}
    We obtain this result using right $\phi$-mixing. However, considering Remarks \ref{Remark:Right-Left} and \ref{Remark:DHY's Theor remarks}, with some modifications this result also holds under left $\phi$-mixing.
\end{remark}

\begin{remark}
    The extremal index $\theta_1$ in Theorem \ref{mainTH} can be replaced by $\theta$ as stated in Remark \ref{Remark:Thetha}. In fact, \cite[Proposition 4.11]{HYang20} states that the Lemma \ref{Lemm2HY} can be applied to the extremal index $\theta$. If $\theta$ exists and we want to state Theorem \ref{mainTH} in terms of that, then the existence of $(\hat\theta_l)_l$ and the condition $\sum_l l \hat\theta_l<\infty$ are not required.
\end{remark}

We use the following lemma several times to prove Theorem \ref{mainTH}. Therefore, we state it here before starting the proof. This lemma states that if our neighbourhood does not have short returns, the extremal index $\theta_1$ equals one.
\begin{lemma}\label{Lemma:DHY_EI=1}
    \cite[Lemma 3.1]{HYang20}. Let $(U_n)_n$ be a sequence of nested sets. Assume that $\pi_{ess}(U_n)\to\infty$ as $n\to \infty$, then $\hat\theta_l$ exists and equals zero for all $l\ge 2$.
\end{lemma}

Now in the next sections, we will prove the theorem. 

\subsection{Proof of Theorem \ref{mainTH}: the case $\alpha=0$ } \vspace{1em}\label{Section: alpha=0}

 By replacing $\alpha$ with 0 in 
 \begin{align*}
    L_{\alpha,s} (\Lambda):=\lim_{n \to \infty} \frac{-1}{s\mu (U_n)^{1-\alpha} } \log(\mu(\tau_r > s\mu (U_n)^{-\alpha})),
\end{align*}
we get the following limit:
\begin{align*}
    L_{0,s} (\Lambda)=\lim_{s \to \infty}\lim_{n \to \infty} \frac{1}{s \mu (U_n)}\log \mu(\tau_{U_n} >s).
\end{align*}

So, our goal in this section is to prove the following:
   \begin{equation}
    L_{0,s} (\Lambda)=\lim_{s \to \infty}\lim_{n \to \infty} \frac{-1}{s \mu (U_n)}\log \mu(\tau_{U_n} >s) = 
    \begin{cases} 
   1 & \text{if $\pi_{ess}(U_n)\to \infty$} \\
   \theta_1 & \text{if $\pi_{ess}(U_n)<\infty$.} 
  \end{cases}\label{Eq=L_0,s}
 \end{equation}

The following lemma facilitates the proof of \eqref{Eq=L_0,s}.
\begin{lemma}\label{Lemma:L_0,s}
    \cite[Lemma 3]{HaydnV20}. Let $(U_n)_n$ be a nested sequence so that $\mu(U_n)\to 0$ as $n\to 0$. Assume that the limits $(\hat{\theta_l})_l$ exist, $\sum_{l=1}^\infty \hat{\theta_l}<\infty$ and $s$ large enough. Then $$  \lim_{s \to \infty}\lim_{n \to \infty} \frac{\mu(\tau_{U_n}\le s)}{s \mu (U_n)}=\theta_1.$$   
\end{lemma}

By Lemma \ref{Lemma:L_0,s} for large enough $s$ we have

$$\mu(\tau_{U_n}\le s)=\theta_1s\mu(U_n)(1+o(1)).$$
    Hence, we obtain the result as below
    \begin{align*}
        \lim_{s \to \infty}\lim_{n \to \infty} \frac{-1}{s \mu (U_n)}\log \mu(\tau_{U_n} >s)&= \lim_{s \to \infty}\lim_{n\to\infty}\frac{-1}{s\mu (U_n)}\log(1-\mu(\tau_{U_n}\le s)\\&
        =\lim_{s \to \infty}\lim_{n\to\infty}\frac{1}{s\mu (U_n)}\mu(\tau_{U_n}\le s)\\&
        =\lim_{s \to \infty}\lim_{n\to\infty}\frac{1}{s\mu (U_n)}\theta_1s\mu(U_n)(1+o(1))\\&
        =\theta_1.
    \end{align*}
    Note that $\lim_{n\to\infty} o(1)=0$. Now by Lemma \ref{Lemma:DHY_EI=1}, the goal of this section i.e. the equality \eqref{Eq=L_0,s} is proved.

\subsection{Proof of Theorem \ref{mainTH}: the case $\alpha\in(0,\infty)$ } \label{Section:Proof of general case}
We consider
\begin{align*}
     L_{\alpha,s} (\Lambda)=\lim_{n \to \infty} \frac{-1}{s\mu (U_n)^{1-\alpha} } \log(\mu(\tau_{U_n} > s\mu (U_n)^{-\alpha}))
\end{align*}
for $\alpha\in(0,\infty)$.

For the proof, the following lemma plays a crucial role. 
\begin{lemma}\label{Lemm2HY}
   \cite[Lemma 4.6]{HYang20}. Let $\mu$ be right $\phi$-mixing for the partition $\mathcal{A}$ with $\phi(k)\le C(k^{-m})$ for some $m>1$. Let $(U_n)_n$, $n=1, 2, ...$ be a good neighbourhood such that $\hat\theta_l(K)$ exists for $K$ large enough, and $\sum_l \hat\theta_l<\infty$. Then we have 
    \begin{align*}
        \frac{\mu(\tau_{U_n}\le s_n)}{s_n\mu(U_n)}\to \theta_1,
    \end{align*}
    for any increasing sequence $s_n, n=1,2,...$ which satisfies $s_n\mu(U_n)\to 0$ as $n\to \infty$.
\end{lemma}

We need the following lemma that gives us a useful inequality, see~\cite[Lemma 4.1]{HYang20} for proof. For convenience, we use $U$ instead of $U_n$.

\begin{lemma}
\cite[Lemma 4.1]{HYang20}. Let $\mu$ be right or left $\phi$-mixing for the partition $\mathcal{A}$. Let $s,t>0$ such that $s<t$ and $U\in \sigma (\mathcal{A}^{\kappa_n})$ and $\Delta < \frac{s}{2}$. Then we have, 
$$
    |\mu(\tau_U>t+s)-\mu(\tau_U>t)\mu(\tau_U>s)|\le 2(\Delta\mu(U)+\phi (\Delta-\kappa_n))\mu(\tau_U>t-\Delta).
$$
\label{ineqlemma}
\end{lemma}
This lemma provides us with an inequality that can be used to derive a more important inequality. In fact, in the proof of the Lemma \ref{k}, this inequality plays a key role. See the proof in \cite[Lemma 4.2]{HYang20}.
Before moving on to the next lemma, for the sake of simplicity, we write 
\begin{equation}
    \delta := 2(\Delta\mu(U)+\phi (\Delta-\kappa_n)).
\end{equation}\label{delta}

A useful observation is that since $\mu(U_n)$ decays either polynomially or exponentially as $n\to\infty$, then for sufficiently large $\Delta$, we have $\delta\le 1$ in both cases.
\begin{lemma}\label{k}
     \cite[Lemma 4.2]{HYang20}. Let $\mu$ be right or left $\phi$-mixing for the partition $\mathcal{A}$. Let $0<\Delta<s<t$. Define $q=\lfloor \frac{s}{\delta}\rfloor$ and $\eta=\frac{q}{q+1}$. Then for every $k\ge 3$ such that $kq\in \Z$ we have
    \begin{equation}\label{rel}
     \mu(\tau_U > ks)\le[\mu(\tau_U > s)+\delta^\eta]^{k-2}.
    \end{equation}
 \end{lemma}
Now we have a relation for $\mu(\tau_U > ks)$ which is obtained from the fact that the measure is $\phi$-mixing. This is the pivotal knowledge we require to accomplish our goal, it means calculating the limit of 
\begin{align*}
    L_{\alpha,s} (\Lambda)=\lim_{n \to \infty} \frac{-1}{s\mu (U_n)^{1-\alpha} } \log(\mu(\tau_r > s\mu (U_n)^{-\alpha}))
\end{align*}

for each $\alpha \in(0,\infty)$ and $s\in\R^+$.

Now we can proceed with the proof of this case. We use the idea used to prove~\cite[Theorem A]{HYang20}. Note that the notation $\mathcal{O}$ in the expression $f=\mathcal{O}(g)$ means that there exists constant $b>0$ such that the inequality $\left|f\right|\le b g$ holds.

In \eqref{rel}, we set $(s_n)_n$ such that this sequence satisfies the conditions of Lemmas \ref{Lemm2HY} since we need this lemma in the proof. Indeed, the left side of $L_{\alpha,s} (z)$ can be obtained by finding a suitable substitution for
$k$ in \eqref{rel}.

So, we put $s_n=\frac{t}{\mu(U_n)^\epsilon}$ for $t\in\R^+$, where $\epsilon\in(0,\min(\alpha,1))$. Also, we replace $k$ by $\frac{1}{\mu(U_n)^{\alpha-\epsilon}}$, then by Lemma \ref{k},
$$  \frac{1}{t\mu (U_n)^{-\alpha} } \log(\mu(\tau_{U_n} > t\mu (U_n)^{-\alpha})) =\frac{\mu (U_n)^{-(\alpha-\epsilon)}-2}{\mu (U_n)^{-(\alpha-\epsilon)}\cdot t\mu (U_n)^{-\epsilon}}\log\left(\mu(\tau_{U_n} > t\mu (U_n)^{-\epsilon})+\mathcal{O}(\delta_n^{\eta_n}))\right).$$
Now, by dividing both sides by $\mu(U_n)$ and taking the limit as $n\to \infty$, we have,
\begin{align*}
    &\lim_{n\to\infty} \frac{1}{t\mu (U_n)^{1-\alpha}} \log(\mu(\tau_{U_n} > t\mu (U_n)^{-\alpha}))\\&=\lim_{n\to\infty} \frac{\mu (U_n)^{-(\alpha-\epsilon)}-2}{\mu (U_n)^{-(\alpha-\epsilon)}\cdot t\mu (U_n)^{-\epsilon}\mu (U_n)}[-\mu(\tau_{U_n} \le t\mu (U_n)^{-\epsilon})+\mathcal{O}(\delta^\eta_n))]\\&
    =\lim_{n\to\infty} \left[\frac{\mu(U_n)^{(\alpha-\epsilon)}}{\mu(U_n)^{(\alpha-\epsilon)}}-2\mu(U_n)^{(\alpha-\epsilon)}\right] \left[\frac{-\mu(\tau_{U_n} \le t\mu (U_n)^{-\epsilon})}{t\mu (U_n)^{-\epsilon}\mu (U_n)}+\frac{\mathcal{O}(\delta^\eta_n)}{t\mu (U_n)^{-\epsilon}\mu (U_n)}\right.
\end{align*}
Thus,
\begin{equation}\label{Equation:terms of L(Lambda)}
    \lim_{n\to\infty} \frac{-1}{t\mu (U_n)^{1-\alpha}} \log(\mu(\tau_{U_n} > t\mu (U_n)^{-\alpha}))
     =\lim_{n\to\infty}\frac{\mu(\tau_{U_n} \le t\mu (U_n)^{-\epsilon})}{t\mu (U_n)^{-\epsilon}\mu (U_n)}-\lim_{n\to\infty}\frac{\mathcal{O}(\delta^\eta_n)}{t\mu (U_n)^{-\epsilon}\mu (U_n)}.
\end{equation}

Now, it suffices to show that the second term on the right-hand side equals 0, then the remaining limit leads us to the desired result with the help of Lemmas \ref{Lemma:DHY_EI=1} and \ref{Lemm2HY}.

Suppose that $b\in(c,1)$, where $c$ is the same as mentioned in (N\ref{Def:adapted1}), such that $\epsilon>b$ and $\Delta_n=\lfloor\mu(U_n)^{-b}\rfloor\gg \kappa_n=o\left(\mu(U_n)^{-c}\right)$. Note that, as we know, $\alpha>\epsilon$ and $b>c$. Since we assume $\alpha>c$, suitable $\epsilon$ and $b$ exist. 

Then we have,
\begin{align*}
    \lim_{n\to\infty}\frac{\mathcal{O}(\delta^{\eta_n}_n)}{t\mu (U_n)^{-\epsilon}\mu (U_n)}&= \lim_{n\to\infty}\frac{\mathcal{O}(\Delta_n^{\eta_n}\mu(U_n)^{\eta_n}+\phi(\Delta_n-\kappa_n)^{\eta_n})}{s_n\mu(U_n)}
    \\&=\lim_{n\to\infty}\frac{\mathcal{O}(\Delta_n^{\eta_n}\mu(U_n)^{\eta_n})}{s_n\mu(U_n)}+\lim_{n\to\infty}\frac{\mathcal{O}(\phi(\Delta_n-\kappa_n)^{\eta_n})}{s_n\mu(U_n)}\\&\lesssim \lim_{n\to\infty}\frac{\Delta_n^{\eta_n}\mu(U_n)^{\epsilon-1+\eta_n}}{t}+\lim_{n\to\infty}\frac{\phi(\Delta_n-\kappa_n)^{\eta_n}\mu(U_n)^{\epsilon-1}}{t}\\&\le\lim_{n\to\infty}\frac{\Delta_n\mu(U_n)^{\epsilon-1+\eta_n}}{t}+\lim_{n\to\infty}\frac{(\Delta_n-\kappa_n)^{-m\eta_n}\mu(U_n)^{\epsilon-1}}{t}\\&\le\lim_{n\to\infty}\frac{\mu(U_n)^{\epsilon-1-b+\eta_n}}{t}+\lim_{n\to\infty}\frac{\mu(U_n)^{bm\eta_n+\epsilon-1}}{t}.
\end{align*}

Now, since $\epsilon>b$, we have $s_n=\frac{t}{\mu(U_n)^\epsilon}\gg\Delta_n=\lfloor\frac{1}{\mu(U_n)^{b}}\rfloor$, then $q_n=\lfloor\frac{s_n}{\Delta_n}\rfloor \to \infty$ as $n\to\infty$. Consequently, $\eta_n=\frac{q_n}{q_{n+1}}\to1$. So, the first term will tend to zero as $n\to\infty$.

To show that the second term tends to zero, the number $bm+\epsilon$ must be strictly greater than one. This condition depends on $\alpha$. If $\alpha>1$, then according to the definition of $\epsilon$, $0<\epsilon<1$ and can be chosen in a way that satisfies $bm+\epsilon>1$ i.e., let $\epsilon=1-\delta'$ and $b=1-2\delta'$ for some $\delta'>0$, then 
$$\epsilon+mb=1-\delta'+m(1-2\delta')=1+m\left(1-2\delta'-\frac{\delta'}{m}\right),$$
hence, $$\epsilon+mb>1 \iff 1-2\delta'-\frac{\delta'}{m}>0 \iff \frac{1}{2+\frac{1}{m}}>\delta'.$$ So in the case $\alpha>1$,  we can choose $\epsilon$ such that $bm+\epsilon>1.$

However, if $\alpha<1$ choosing $\epsilon$ to satisfy the condition may be problematic, as $\alpha$ could be smaller than $1-mb$. Now, we want to determine how small $\alpha$ it can be for condition $bm+\epsilon>1$ to hold. In this case, we have $0<\epsilon<\alpha$. Now suppose that $\epsilon=\alpha-\delta''$ and $b=\alpha-2\delta''$ for some $\delta''>0$. Note that condition $b>c$ holds, if $\alpha-2\delta''=b>c$, i.e., $\delta''<\frac{\alpha-c}{2}$, so a suitable $\delta''$ exists if $\alpha>c$. Now let $$bm+\epsilon=(\alpha-2\delta'')m+\alpha-\delta''=\alpha(m+1)-\delta''(2m+1).$$ Thus, $$bm+\epsilon=\alpha(m+1)-\delta''(2m+1)>1 \iff \delta''<\frac{\alpha(m+1)-1}{2m+1} \iff \alpha>\frac{1}{m+1}.$$ 

Since we assume $\alpha>\frac{1}{m+1}$, the second term also tends to zero, so
$$ L_{\alpha,s} (\Lambda)=\lim_{n\to\infty}\frac{\mu(U_n\le s_n)}{s_n\mu(U_n)}, $$ since $s_n=\frac{t}{\mu(U_n)^\epsilon}$ satisfies in the condition of Lemma \ref{Lemm2HY}, we have $$ L_{\alpha,s} (\Lambda)=\theta_1=1-\hat{\theta_2},$$ if $\pi_{ess}(U_n)<\infty$. In the case $\pi_{ess}(U_n)\to\infty$ by Lemma \ref{Lemma:DHY_EI=1}$,\hat{\theta_2}=0$ so $L_{\alpha,s} (\Lambda)=1$.

So, we have reached the desired conclusion and the proof of Theorem \ref{mainTH} is complete.

%\begin{remark}
  %  As observed in the proof of the polynomial and exponential cases, where we want to show $\lim_{n\to\infty}\frac{\mathcal{O}(\delta^\eta_n)}{s_n\mu (U_n)}=0$, a lower bound for $\mu(U_n)$ is essential. When $\mu(U_n)$ appears in the denominator of a fraction, having a lower bound ensures that the fraction is well-defined in $n\to \infty$.
%\end{remark}
\section{recurrence rate for Gibbs-Markov systems}
Our goal in this section is to apply Theorem \ref{mainTH} to Gibbs-Markov systems. Before that, we give some definitions.

Let $T:\Omega\to\Omega$ be a map, where $\Omega$ is a compact metric space, and let $\mu$ be an invariant probability measure on $\Omega$.

\begin{definition}
    We call $T$ a \textit{Markov} map if:
    \begin{enumerate}
        \item There exists a countable measurable partition $\mathcal{A}$ on $\Omega$ such that for all $A\in\mathcal{A}$, $\mu(A)>0$,
        \item $T(A)$ can be written as a union of elements in $\mathcal{A}$,
        \item $T:A\to T(A)$ is bijective for all $A\in\mathcal{A}$.
    \end{enumerate}
\end{definition}
Let $\mathcal{A}^n=\bigvee^{n-1}_{i=0} T^{-i}\mathcal{A}$ be as in Definition \ref{Def:cylinders} and assume $\lambda \in (0,1)$. We define a metric on $\Omega$ that depends on $\lambda$:
$$d_{\lambda}(x,y):=\lambda^{s(x,y)},$$ where $s(x,y)$ is the largest position integer $n$ such that $x$ and $y$ lie in the same $n$-cylinder.

Define the Jacobian $g=\frac{1}{\lvert detJT\rvert}=\frac{d\mu}{d\mu\circ T}$ and $g_k(x)=g(x) g\circ T(x)\cdots g\circ T^{k-1}(x)$. Note that since $T$ is injective, $JT$ shows how the density of the measure changes under $T$.

\begin{definition}\label{Definition:Gibbs-Markov map}
    Let $T$ be a Markov map. We call $T$ as a \textit{Gibbs-Markov} map with respect to $\mu$ if:
    \begin{enumerate}
        \item Big image probability(BIP): There exists $C>0$ such that $\mu(T(A))>C$ for all $A\in\mathcal{A}$,
        \item Distortion: The function $\log g\lvert_A:\Omega\to\R$ is Lipschits for all $A\in\mathcal{A}$ i.e., there exists $C'>0$ such that $d_{\R}\left(\log g\lvert_A(x), \log g\lvert_A(y)\right)\le C' d_{\lambda}(x,y)$.
    \end{enumerate}
\end{definition}
\begin{definition}
    We call $(T,\mu,\mathcal{A})$ a \textit{Gibbs-Markov system} if $T$ is a Gibbs-Markov map with respect to $\mu$.
\end{definition}
\begin{notation}
    Let $A_n(x)\in\mathcal{A}^n$ be the $n$-cylinder that contains the point $x$.
\end{notation}
\begin{remark}
    \begin{enumerate}
        \item Distortion bound property. In view of BIP and distortion condition of Definition \ref{Definition:Gibbs-Markov map}, there exists $D_1>1$ such that for all $x,y$ in the same $n$-cylinder, we have the following distortion bound: $$\left\lvert\frac{g_n(x)}{g_n(y)}-1\right\rvert\le D_1\ d_{\lambda}\left(T^n(x),T^n(y)\right).$$
        \item Gibbs property. For some constant $D_2$, $$D_2^{-1}\le\frac{\mu\left(A_n(x)\right)}{g_n(x)}\le D_2.$$
    \end{enumerate}
\end{remark}
The next lemma shows that Gibbs-Markov systems are exponentially $\phi$-mixing systems. The proof can be found in \cite[Lemma 2.4]{MelbournNicol05}.
\begin{lemma}\label{Lemma: Gibbs-Markov is phi mixing}
    \cite[Lemma 2.4]{MelbournNicol05}. Let $(T,\mu,\mathcal{A})$ be a Gibbs-Markov system with respect to the partition $\mathcal{A}$. Then $\mu$ is a $\phi$-mixing measure such that $\phi$ decays exponentially, i.e. $\phi (k)\le \zeta^k$ for some $\zeta\in(0,1)$.

\end{lemma}

    The main theorem of this section is the following.
\begin{theorem}\label{Theorem:Gibbs-Markov}
     Let $(T,\mu,\mathcal{A})$ be a Gibbs-Markov system. Assume that $(U_n)_n$ is a good neighbourhood system such that $\bigcap_nU_n=\{\Lambda\}$ where $\Lambda$ is a measure zero set. If $\pi_{ess}(U_n)<\infty$, then assume that $(\hat\theta_l)_l$ exists, and satisfies $\sum_l l \hat\theta_l<\infty$. Let $\alpha\in(c,\infty)$, where $c$ as defined in (N\ref{Def:adapted1}). Then for any $s\in\R^+$, $\theta_1$ exists and,
\[  L_{\alpha,s} (\Lambda) = 
   \begin{cases} 
   1 & \text{if $\pi_{ess}(U_n)\to \infty$} \\
   \theta_1 & \text{if $\pi_{ess}(U_n)<\infty$.} 
  \end{cases}
\]
\end{theorem}

\begin{remark}
    Theorem \ref{Theorem:Gibbs-Markov} can also be stated in terms of $\theta$, which is defined in Remark \ref{Remark:Thetha}. Yang proved in \cite[Proposition 5.4]{Yang} that under the conditions of Theorem \ref{Theorem:Gibbs-Markov}, we have $\theta=\theta_1$. If we want to state Theorem \ref{Theorem:Gibbs-Markov} in terms of $\theta$, then the existence of $(\hat\theta_l)_l$ and the condition $\sum_l l \hat\theta_l<\infty$ are not required.
\end{remark}
\begin{remark}
    This result also holds for $\phi$-mixing systems with exponential decay. What can be said is that if the dynamical system mixes sufficiently quickly, then the condition on $\alpha$ becomes simply $c<\alpha$.
\end{remark}
\begin{remark}
    The case $\alpha=\infty$ is proved in \cite[Theorem B]{HYang20}.
\end{remark}
\begin{remark}
    Theorem \ref{Theorem:Gibbs-Markov} holds for $\alpha=0$, the proof is the same as in Section \ref{Section: alpha=0}. Note that in this case, only the following conditions need to exist, $(\hat\theta_l)_l$ exists, and $\sum_l  \hat\theta_l<\infty$ where $\pi_{ess}(U_n)<\infty$.
\end{remark}

The idea of proof of Theorem \ref{Theorem:Gibbs-Markov} is exactly the same as what we used in the proof of Section \ref{Section:Proof of general case}. We need a lemma similar to the Lemma \ref{Lemm2HY} for Gibbs-Markov systems. This lemma was proved by Davis, Haydn and Yang in \cite[Lemma 4.9]{HYang20}.

\begin{lemma}\label{Lemma: equality of theta and fractal in Gibbs-M}
   \cite[Lemma 4.9]{HYang20}. Let $(T,\mu,\mathcal{A})$ be a Gibbs-Markov system. Assume that $(U_n)_n$ is a good neighbourhood system such that $\hat\theta_l(K)$ exists for $K$ large enough, and $\sum_l  \hat\theta_l<\infty$. Then we have
    \begin{align*}
        \frac{\mu(\tau_{U_n}\le s_n)}{s_n\mu(U_n)}\to \theta_1,
    \end{align*}
    for any increasing sequence $(s_n)_n$ for which $s_n\mu(U_n)\to 0$ as $n\to \infty$.
\end{lemma}
\begin{proof}[Proof of Theorem \ref{Theorem:Gibbs-Markov}]

Let $\alpha\in(c,\infty)$. Put $s_n=\frac{t}{\mu(U_n)^\epsilon}$ for $t\in\R^+$ and $\epsilon\in(0,\min(1,\alpha))$. Given Lemma \ref{Lemma: Gibbs-Markov is phi mixing}, we have $\phi (k)\le \zeta^k$ for some $\zeta\in(0,1)$. Therefore, we can deduce relation \eqref{Equation:terms of L(Lambda)} for Gibbs-Markov systems. Thus,
$$ \lim_{n\to\infty} \frac{-1}{t\mu (U_n)^{1-\alpha}} \log(\mu(\tau_{U_n} > t\mu (U_n)^{-\alpha}))
     =\lim_{n\to\infty}\frac{\mu(\tau_{U_n} \le t\mu (U_n)^{-\epsilon})}{t\mu (U_n)^{-\epsilon}\mu (U_n)}-\lim_{n\to\infty}\frac{\mathcal{O}(\delta^\eta_n)}{t\mu (U_n)^{-\epsilon}\mu (U_n)}.$$

Now it is enough to show that second term tends to zero, then by Lemma \ref{Lemma: equality of theta and fractal in Gibbs-M} and Lemma \ref{Lemma:DHY_EI=1}, we obtain the result.

Same as proof of Theorem \ref{Theorem:MainBDT}, suppose numbers $b\in(c,1)$ such that $\epsilon>b$ and $\Delta_n=\lfloor\mu(U_n)^{-b}\rfloor\gg \kappa_n=o\left(\mu(U_n)^{-c}\right)$, where $c$ is the same as in (N\ref{Def:adapted1}). We assume $\alpha>c$, to have allow chose $b$ such that $\epsilon>b>c$. Now, we show  $\lim_{n\to\infty}\frac{\mathcal{O}(\delta^\eta_n)}{t\mu (U_n)^{-\epsilon}\mu (U_n)}=0$ as bellow:
 \begin{align*}
    \lim_{n\to\infty}\frac{\mathcal{O}(\delta^{\eta_n}_n)}{t\mu (U_n)^{-\epsilon}\mu (U_n)}&= \lim_{n\to\infty}\frac{\mathcal{O}(\Delta_n^{\eta_n}\mu(U_n)^{\eta_n}+\phi(\Delta_n-\kappa_n)^{\eta_n})}{s_n\mu(U_n)}
    \\&=\lim_{n\to\infty}\frac{\mathcal{O}(\Delta_n^{\eta_n}\mu(U_n)^{\eta_n})}{s_n\mu(U_n)}+\lim_{n\to\infty}\frac{\mathcal{O}(\phi(\Delta_n-\kappa_n)^{\eta_n})}{s_n\mu(U_n)}\\&\lesssim \lim_{n\to\infty}\frac{\Delta_n^{\eta_n}\mu(U_n)^{\epsilon-1+\eta_n}}{t}+\lim_{n\to\infty}\frac{\phi(\Delta_n-\kappa_n)^{\eta_n}\mu(U_n)^{\epsilon-1}}{t}\\&\le \lim_{n\to\infty}\frac{\Delta_n\mu(U_n)^{\epsilon-1+\eta_n}}{t}+\lim_{n\to\infty}\frac{\zeta^{(\Delta_n-\kappa_n){\eta_n}}\mu(U_n)^{\epsilon-1}}{t}\\& \le \lim_{n\to\infty}\frac{\mu(U_n)^{\epsilon-1-b+\eta_n}}{t}+\lim_{n\to\infty}\frac{\zeta^{\kappa_n{\eta_n}}\mu(U_n)^{\epsilon-1}}{t}.
\end{align*}
Since $\epsilon>b$, $\eta_n\to 1$ as $n\to\infty$ and the first term tends to zero. To show that the second term ternds to zero, as $\zeta^{\kappa_n}\to 0$ and $\mu(U_n)^{\epsilon-1}\to \infty$ as $n\to\infty$, the condition is that the growth of $\zeta^{\kappa_n}$ must be greater than $\mu(U_n)^{\epsilon-1}$. Thus, we have,
$$\zeta^{\kappa_n}> \mu(U_n)^{\epsilon-1} \quad \forall n\iff \frac{\kappa_n}{\log(\mu(U_n))}>\frac{\epsilon-1}{\log\zeta} \quad \forall n.$$
Since $\log(\mu(U_n)$ is a negative number we have,
$$ \frac{\kappa_n}{\mu(U_n)^{-c}}>\frac{\kappa_n}{\log(\mu(U_n))}>\frac{\epsilon-1}{\log\zeta}.$$ Now as $\kappa_n=o\left(\mu(U_n)^{-c}\right)$, we have $\lim_{n\to\infty}\frac{\kappa_n}{\mu(U_n)^{-c}}=0$. Hence, $$\lim_{n\to\infty}\frac{\zeta^{\kappa_n{\eta_n}}\mu(U_n)^{\epsilon-1}}{t}=0 \iff \epsilon<1.$$
Since $\epsilon\in(0,\min(1,\alpha))$, the above condition holds and then $$  \lim_{n\to\infty}\frac{\mathcal{O}(\delta^{\eta_n}_n)}{t\mu (U_n)^{-\epsilon}\mu (U_n)}=0,$$ so $$ L_{\alpha,s} (\Lambda)=\lim_{n\to\infty}\frac{\mu(\tau_{U_n}\le s_n)}{s_n\mu(U_n)}.$$ Since $(s_n)_n$ defined in a way that satisfies the conditions of Lemma \ref{Lemma: equality of theta and fractal in Gibbs-M}, therefore, considering this lemma as well as Lemma \ref{Lemma:DHY_EI=1}, we have 
\[  L_{\alpha,s} (\Lambda) = 
   \begin{cases} 
   1 & \text{if $\pi_{ess}(U_n)\to \infty$} \\
   \theta_1 & \text{if $\pi_{ess}(U_n)<\infty$.} 
  \end{cases}
\]
Therefore, the desired result is obtained, and the proof of the theorem is completed.\end{proof}

\section{recurrence rate and metric balls in a $\phi$-mixing system}

In this section, we give applications of Theorem \ref{mainTH}. Using metric balls, we determine the recurrence rate in a metric space and find the limit $ L_{\alpha,s} (\Lambda)$. Indeed, we explain the concept of exceedance rate and then describe its connection to the recurrence rate.
Davis, Haydn and Yang have previously done in~\cite[Section 5]{HYang20} for the case $\alpha=\infty$(escape rate). This motivated us to derive the general formula in this space as well. So, the idea of proof of Theorem \ref{Theorem:B_r} is similar to the work done by \cite{HYang20}, and we achieve our desired result with modifications to their proof.

Assume that $\Omega$ is a metric space, $T:\Omega\to \Omega$ is a map on $\Omega$, and $\mu$ is the probability measure. Let $g:\Omega\to \R\cup\{\infty\}$ be a continuous function such that the maximal value of $g$ is achieved on $\Lambda$, a measure zero closed set. 

Such a function is called an \textit{observable} in dynamical systems. Now, we consider the following process:
$$X_0:=g,\quad X_1:=g\circ T, \cdots,\   X_k:=g\circ T^k, \ldots.$$ Let $(u_n)_n$ be a non decreasing sequence of real numbers. We consider $u_n$ as a sequence of thresholds and events $\{X_n>u_n\}$ represents an instance where $X_k$ exceeds the threshold $u_n$. Define 
\begin{equation}\label{Def:Open balls}
U_n:=\{X_0>u_n\}
\end{equation} 
as the open sets. Since $U_{n+1}\subset U_n$ and $\lim_{n\to\infty}\mu(U_n)\to0,$ $(U_n)_n$ is a nested sequence of open sets.

Now, if we put $M_n:=\max\{X_k:k=0,\cdots,n-1\}$, one can calculate is $$\zeta(u_n):=\lim_{t\to\infty}\frac{1}{t}|\log\mu(M_t<u_n)|.$$
\begin{definition}
    Let $(u_n)_n$ be a sequence of non decreasing of real numbers. Then $$\zeta(g,(u_n)_n):=\lim_{n\to\infty}\frac{\zeta(u_n)}{\mu(U_n)},$$ is called \textit{exceedance rate} of $g$ along the thresholds $(u_n)_n$, if the limit exists.
\end{definition}

The following introduces the definition of approximation close to the $U_n$.
\begin{definition}
    Let $r_n>0$, then $$U_n^-:=U_n\setminus\overline{B_{r_n}(U_n)}=U_n\setminus\left(\bigcup_{x\in\partial U_n}\overline{B_{r_n}(x)}\right),$$
\end{definition}
is \textit{the inner approximation for $U_n$}. Similarly, \textit{the outer approximation for $U_n$} is defined by $$U_n^+:=B_{r_n}(U_n)=\bigcup_{x\in U_n}B_{r_n}(x).$$
\begin{remark}
    $\overline{U_n^-}\subset U_n$ and $\overline{U_n}\subset U_n^+.$
\end{remark}
Now, we can state a theorem similar to Theorem \ref{mainTH} for metric balls as follows.

\begin{theorem}\label{Theorem:B_r}
    Let $(\Omega,T,\mu)$ be a $\phi$ mixing dynamical system where $T$ is measure preserving and $\mu$ is a right $\phi$-mixing for partition $\mathcal{A}$ of $\Omega$ where $\phi$ decays at least polynomially with power $m$.
    Let $g:\Omega\to \R\cup\{\infty\}$ be a continuous function achieving its maximum on a measure zero set $\Lambda$. Let $(u_n)_n$ be a non decreasing sequence of real numbers with $u_n\nearrow\sup g$, such that open sets $U_n$ defined in \eqref{Def:Open balls} satisfy in $\mu(U_n^+\setminus U_n^-)=o(1)\mu(U_n)$ for a sequence $r_n\to 0$. If $\pi_{ess}(U_n)<\infty$, then assume that $(\hat\theta_l)_l$ defined in \eqref{Def:theta_l} exists, and satisfies $\sum_l l \hat\theta_l<\infty$. Let $\kappa_n$ be the smallest positive integer for which diam $(\mathcal{A}^{\kappa_n}) \le r_n$ and assume:
    \begin{enumerate}
    \item $\kappa_n\mu(U_n)^{c'}\to 0$ for some $c'\in(0,1)$;
    \item $U_n$ has small boundary: there exists $C>0$ and $m'>1$ so that $\mu(\bigcup_{B\in\mathcal{A}^j,B\cap B_{r_n}(\partial U_n)\ne \emptyset}B)\le Cj^{-m'}$ for all $n$ and $j\le \kappa_n$.
\end{enumerate}
   Let $\alpha$ be a number in $[0,\infty)$ such that for $\alpha\in(0,1)$, we have $\alpha>\max\{\frac{1}{m+1},c'\}$. Then for all $s\in\R^+$, $\theta_1$ exists and,
 \[  L_{\alpha,s} (\Lambda) = 
   \begin{cases} 
   1 & \text{if $\pi_{ess}(U_n)\to \infty$} \\
   \theta_1 & \text{if $\pi_{ess}(U_n)<\infty$.} 
  \end{cases}
\]
\end{theorem}
\begin{remark}
    Since $\{M_t<u_n\}=\{\tau_{U_n}>t\},$ then 
    
    \[  \zeta(g,\{u_n\}_n)=\rho(\Lambda)=L_{\infty,s} (\Lambda) = 
   \begin{cases} 
   1 & \text{if $\pi_{ess}(U_n)\to \infty$} \\
   \theta_1 & \text{if $\pi_{ess}(U_n)<\infty$,} 
  \end{cases}
\]that is proved in in~\cite[Theorem E]{HYang20}.
\end{remark}
\begin{remark}\label{Remark: B_r for Gibbs-Markov}
    A similar result holds for Gibbs-Markov systems and $\phi$-mixing dynamical systems with an exponentially rate. It suffices to modify condition $\alpha$ according to what is stated in Theorem \ref{Theorem:Gibbs-Markov}. The proof of this follows the same reasoning as the one given below.
\end{remark}
\begin{remark} \label{remark: thete=theta1 in B_r}
    Similar to the previous results, in Theorem \ref{Theorem:B_r}, $\theta$ stated in Remark \ref{Remark:Thetha} can also be replaced with $\theta_1$. However, in this case, the existence of $(\hat\theta_l)_l$ and the condition $\Sigma_ll\hat{\theta_l}<\infty$ are no longer required, and condition (ii) replaced with the following condition: 
    \begin{center}
        (2)$'$ $\theta$ exists for some sequence $K_n>\kappa_n^2.$
    \end{center}
\end{remark}
The proof idea is as follows: we use cylinder $U_n$ and its inner and outer approximations to find two sets bounded $U_n$ and satisfy the conditions of Definition \ref{adapt} and then those of Theorem \ref{mainTH}, we then apply the result of Theorem \ref{mainTH} to both bounds, which finally establishes the result of Theorem \ref{mainTH} for $U_n$ that lies between these two sets.

To proving Theorem \ref{Theorem:B_r}, we need the following lemma and proposition. We use the next notations to make statements simpler.
\begin{notation}
    \begin{enumerate}
        \item We use $\hat\theta^{U_n}_l$ as the notation for $\hat\theta_l$ with respect to $(U_n)_n$ i.e.
        
        $$\hat\theta_l^{U_n}=\lim_{K\to\infty}\lim_{n\to\infty}\mu_{U_n}\left(\tau^{l-1}_{U_n}\le K\right).$$
        \item We use $\theta^{U_n}_l$ as the notation for $\theta_l$ with respect to $U_n$ i.e.
        
        $$\theta_l^{U_n}=\lim_{K\to\infty}\lim_{n\to\infty}\mu_{U_n}\left(\tau^{l-1}_{U_n}> K\right).$$
    \end{enumerate}
\end{notation}
The next lemma shows that the extremal index in a $\phi$-mixing dynamical system is independent of the choice of the cylinder, which was proved in \cite[Lemma 5.6]{Yang}.
\begin{lemma}\label{Lemma:equal thetas}
    \cite[Lemma 5.6]{Yang}. Let $(U_n)_n$ and $(V_n)_n$ be two sequences of nested sets with $V_n<U_n$ for each $n$. Then $\hat\theta_l^{U_n}=\hat\theta_l^{V_n}$.
\end{lemma}
\begin{remark}
    Due to Lemma \ref{Lemma:equal thetas}, we have $\theta_l^{U_n}=\theta_l^{V_n}$, for $(U_n)_n$ and $(V_n)_n$ as two sequences of nested sets with $V_n<U_n$ for each $n$.
\end{remark}

Consider the following two sets, 
\begin{equation}\label{Def: V_n,W_n}
    V_n=\bigcup_{B\in\mathcal{A}^{\kappa_n}, B\subset U_n}B, \qquad W_n=\bigcup_{B\in\mathcal{A}^{\kappa_n}, B\cap U_n\ne \emptyset}B.
\end{equation}

\begin{remark}\label{Remark:Inequality}
  $V_n\subset U_n\subset W_n$. Then, for $\alpha\in[0,1)$, $\mu^{1-\alpha}(V_n)\le \mu^{1-\alpha}(U_n)\le \mu^{1-\alpha}(W_n)$ and for $\alpha\in(1,\infty)$, $\mu^{1-\alpha}(W_n)\le \mu^{1-\alpha}(U_n)\le \mu^{1-\alpha}(V_n)$.
 \end{remark}

 $V_n$ is the largest union of $\kappa_n$-cylinders contained in $U_n$ and $W_n$ is the smallest union of $\kappa_n$-cylinders that contains $U_n$. So, we have the nested sequences $(V_n)_n$ and $(W_n)_n$.

  The following proposition shows that $(V_n)_n$ and $(W_n)_n$ are good neighbourhood systems which is proved in \cite[Theorem E]{HYang20}.

\begin{proposition}\label{Lemma:adapted_U_n}
    Let $\Omega$ be a compact metric space and $T:\Omega\to \Omega$ a measure preserving map. Suppose that $\mu$ is an invariant probability measure on $\Omega$. Let $\mathcal{A}$ be a partition and the sequences $(V_n)_n$ and $(W_n)_n$ defined as \eqref{Def: V_n,W_n}. Let $\kappa_n$ be the smallest positive integer for which diam $(\mathcal{A}^{\kappa_n}) \le r_n$ and assume:
    \begin{enumerate}
    \item $\kappa_n\mu(U_n)^{c'}\to 0$ for some $c'\in(0,1)$;
    \item $U_n$ has small boundary: there exists $C>0$ and $m'>1$ so that $\mu(\bigcup_{B\in\mathcal{A}^j,B\cap B_{r_n}(\partial U_n)\ne \emptyset}B)\le Cj^{-m'}$ for all $n$ and $j\le \kappa_n$.
\end{enumerate} Then $(V_n)_n$ and $(W_n)_n$ are good neighbourhood systems .
\end{proposition}

Now we can prove our main theorem in this section. As previously mentioned, the proof idea comes from~\cite[Section 5]{HYang20}.

\begin{proof}[Proof of Theorem \ref{Theorem:B_r}]
   
    Note that for $V_n\subset U_n\subset W_n$ and each $\alpha\in [0,\infty)$ we have
    $$\mu (V_n)^{\alpha}\tau_{W_n}\le \mu (U_n)^{\alpha}\tau_{U_n}\le \mu (W_n)^{\alpha}\tau_{V_n}.$$
    Fix $s\in\R^+$, then as random variables, they are stochastically ordered i.e. 
    $$\mu\left(\mu (V_n)^{\alpha}\tau_{W_n} > s\right)\le \mu\left(\mu (U_n)^{\alpha}\tau_{U_n} > s\right) \le \mu\left(\mu (W_n)^{\alpha}\tau_{V_n} > s\right).$$
    Now by Remark \ref{Remark:Inequality} for $\alpha\in[0,1)$ we have,
    \begin{align*}
    &\frac{-1}{s\mu (V_n)^{1-\alpha} } \log\left(\mu(\tau_{V_n} > s\mu (W_n)^{-\alpha})\right)  
    \le\frac{-1}{s\mu (U_n)^{1-\alpha} } \log\left(\mu(\tau_{U_n} > s\mu (U_n)^{-\alpha})\right) \\&
    \le\frac{-1}{s\mu (W_n)^{1-\alpha} } \log\left(\mu(\tau_{W_n} > s\mu (V_n)^{-\alpha})\right),
    \end{align*}

    and similarly for $\alpha\in[1,\infty)$ one can obtain,
 \begin{align*}
    &\frac{-1}{s\mu (W_n)^{1-\alpha} } \log\left(\mu(\tau_{V_n} > s\mu (W_n)^{-\alpha})\right)
    \le\frac{-1}{s\mu (U_n))^{1-\alpha} } \log\left(\mu(\tau_{U_n} > s\mu (U_n)^{-\alpha})\right) \\&
    \le\frac{-1}{s\mu (V_n)^{1-\alpha} } \log\left(\mu(\tau_{W_n} > s\mu (V_n)^{-\alpha})\right).
    \end{align*}

   As $U_n^-\subset V_n$, $W_n\subset U_n^+$, we have $W_n\setminus V_n\subset U_n^+\setminus U_n^-$ so by assumption $\mu(U_n^+\setminus U_n^-)=o(1)\mu(U_n)$, we have $\mu(V_n)=\mu(W_n)(1+o(1))$.
 So, for $\alpha\in[0,\infty)$, $$\mu(V_n)^{\alpha}=\mu(W_n)^{\alpha}(1+o(1))^{\alpha}\quad as \quad n\to \infty,$$
$$\mu(V_n)^{1-\alpha}=\mu(W_n)^{1-\alpha}(1+o(1))^{1-\alpha}\quad as \quad n\to \infty.$$
Thus, we have:
$$
    \lim_{n\to \infty} \frac{-1}{s\mu (V_n)^{1-\alpha} } \log\left(\mu(\tau_{V_n} > s\mu (W_n)^{-\alpha})\right)
    =\lim_{n\to \infty} \frac{-1}{s\mu (V_n)^{1-\alpha} } \log\left(\mu(\tau_{V_n} > s\mu (V_n)^{-\alpha})\right),
$$
and also,
$$
    \lim_{n\to \infty}\frac{-1}{s\mu (W_n)^{1-\alpha} } \log\left(\mu(\tau_{W_n} > s\mu (V_n)^{-\alpha})\right)
    =\lim_{n\to \infty}\frac{-1}{s\mu (W_n)^{1-\alpha} } \log\left(\mu(\tau_{W_n} > s\mu (W_n)^{-\alpha})\right).
$$
Similarly, for both bounds in case $\alpha\in[1,\infty)$, as $n\to \infty$, we can replace $\mu (W_n)^{-\alpha}$ with $\mu (V_n)^{-\alpha}$ and $\mu (V_n)^{1-\alpha}$ with $\mu (W_n)^{1-\alpha}$.
   Thus, in both cases of $\alpha$, we have the following inequality:
     \begin{align}
    &\lim_{n\to\infty}\frac{-1}{s\mu (V_n)^{1-\alpha} } \log\left(\mu(\tau_{V_n} > s\mu (V_n)^{-\alpha})\right)\notag \\&
    \le \lim_{n\to\infty}\frac{-1}{s\mu (U_n)^{1-\alpha} } \log\left(\mu(\tau_{U_n} > s\mu (U_n)^{-\alpha})\right) \tag{3.3}\label{Inequality: boundaries of metric balls} \\&
    \le \lim_{n\to\infty}\frac{-1}{s\mu (W_n)^{1-\alpha} } \log\left(\mu(\tau_{W_n} > s\mu (W_n)^{-\alpha})\right).\notag
    \end{align}
    Since $(V_n)_n$ and $(W_n)_n$ are good neighbourhood systems by Proposition \ref{Lemma:adapted_U_n}, if for the case $\alpha<1$ we have $\alpha>\max\{\frac{1}{m+1},c'\}$, then we can apply Theorem \ref{mainTH} for two bounds in \eqref{Inequality: boundaries of metric balls}, so for all $s\in\R^+$ we have 
\[ \lim_{n\to\infty} \frac{-1}{s\mu (V_n)^{1-\alpha} } \log(\mu(\tau_{V_n} > s\mu (V_n)^{-\alpha})) = \begin{cases} 
   1 & \text{if $\pi_{ess}(U_n)\to \infty$} \\
   \theta_1 & \text{if $\pi_{ess}(U_n)<\infty$,} 
  \end{cases}
\]
  and 

\[ \lim_{n\to\infty} \frac{-1}{s\mu (W_n)^{1-\alpha} } \log(\mu(\tau_{W_n} > s\mu (W_n)^{-\alpha})) = \begin{cases} 
   1 & \text{if $\pi_{ess}(U_n)\to \infty$} \\
   \theta_1 & \text{if $\pi_{ess}(U_n)<\infty$.} 
  \end{cases}
\]

  By Lemma \ref{Lemma:equal thetas}, since $\theta_1^{W_n}=\theta_1^{V_n}$,  then $\theta_1^{U_n}=\theta_1^{W_n}=\theta_1^{V_n}=\theta_1$. Hence,
   \[  L_{\alpha,s} (\Lambda) = \lim_{n\to\infty} \frac{-1}{s\mu (U_n)^{1-\alpha} } \log(\mu(\tau_{U_n} > s\mu (U_n)^{-\alpha}))=
   \begin{cases} 
   1 & \text{if $\pi_{ess}(U_n)\to \infty$} \\
   \theta_1 & \text{if $\pi_{ess}(U_n)<\infty$.} 
  \end{cases}
\]
Therefore, the required result is obtained, and the proof of the Theorem \ref{Theorem:B_r} is completed. \end{proof}

\section{Examples}

In the final section of this paper, we present several examples. The first one examines the case where $\Lambda$ is a singleton. If the point is non-periodic, $ L_{\alpha,s} (\Lambda)$ is equal to one, and if the point is periodic, $ L_{\alpha,s} (\Lambda)$ is less than one. The second example considers $\Lambda$ as the ternary Cantor set, since we have short returns, the recurrence rate is less than one. In the last example, we assume $\Lambda$ is a submanifold of an Anosov diffeomorphism, where the recurrence rate varies depending on its stable and unstable vectors. 
\subsection{$\Lambda$ as a singleton}
Here we consider the case where $\Lambda=\{z\}$, for $z\in\Omega$. As discussed in Section \ref{Section: Mike Theorem},  $ L_{\alpha,s} (z)$ has been derived for interval maps in \cite[Theorem 2.1]{BDTodd18}. Also $ L_{\infty,s}(z)$ was found for $\phi$-mixing systems in \cite[Theorem 1]{HaydnY20}. However here we aim to show that a singleton is a special case of our results.

In such a setting, a sequence of sets is considered to shrink to $z$. The important aspect of these settings is the extremal index. In the \cite{HaydnY20}, the extremal index for a periodic point $z$ is defined as follows:
$$\theta=\theta(z)=\lim_{n\to\infty}\frac{\mu(U_n\cap T^{-p}U_n)}{\mu(U_n)},$$ where $p$ is the period of $z$. Note that for non-periodic point $z$ we have $\theta=0$. Also in this case we have $\pi_{ess}(U_n)\to \infty$ by following lemma.
\begin{lemma}
    \cite[Lemma 1]{HaydnY20}. Let $\mathcal{A}$ be a partiotion of $\Omega$ and $(U_n)_n$ a decreasing sequence shrinking to $z$. Then $(\tau(U_n))_n$ is bounded if and only if $z$ is a periodic point.
\end{lemma}

On the other hand, for non-periodic points, we have $\pi_{ess}(U_n)<\infty$ and \cite[Section 8.3]{HaydnV20} shows that $\Sigma_ll\hat{\theta_l}<\infty$ and $\theta=\theta_1=1$. Hence we have Theorem \ref{mainTH} for case $\Lambda=\{z\}$:

\begin{theorem}
      Let $(\Omega, T,\mu)$ be a $\phi$-mixing dynamical system with a measurable partition $\mathcal{A}$ of $\Omega$ such that $\phi$ decays at least polynomially with a power $m>1$. Assume that $T:\Omega\circlearrowleft$ preserves $\mu$. Let $z\in\Omega$ and suppose $(U_n)_n$ is a good neighbourhood system which shrinks to $z$ such that $\bigcap_nU_n=\{z\}$. Let $\alpha \in [0,\infty)$ such that for $\alpha\in(0,1)$, we have $\alpha>\max\{\frac{1}{m+1},c\}$, where $c$ as defined in (N\ref{Def:adapted1}). Then for any $s\in\R^+$, $\theta=\lim_{n\to\infty}\frac{\mu(U_n\cap T^{-p}U_n)}{\mu(U_n)}$ exists and,
\[  L_{\alpha,s} (z) = 
   \begin{cases} 
   1 & \text{if $z$ is non-periodic,} \\
   1-\theta & \text{if $z$ is periodic.} 
  \end{cases}
\]

\end{theorem}

    And also for Gibbs-Markov systems we have:

\begin{theorem}
      Let $(\Omega, T,\mu)$ be a Gibbs-Markov system. Let $z\in\Omega$ and suppose $(U_n)_n$ is a good neighbourhood system which shrinks to $z$ such that $\bigcap_nU_n=\{z\}$. Let $\alpha \in [c,\infty)$, where $c$ as defined in (N\ref{Def:adapted1}). Then for any $s\in\R^+$, $\theta=\lim_{n\to\infty}\frac{\mu(U_n\cap T^{-p}U_n)}{\mu(U_n)}$ exists and,
\[  L_{\alpha,s} (z) = 
   \begin{cases} 
   1 & \text{if $z$ is non-periodic,} \\
   1-\theta & \text{if $z$ is periodic.} 
  \end{cases}
\]
\end{theorem}

Now, we want to determine the recurrence rate in a metric space using metric balls and find the limit $ L_{\alpha,s} (z)$. Suppose that $U_n=B_r(z)$ and $g(y)=h(d(y,z))$ is the observable for some function $h(x):\R\to\R\cup\{\infty\}$ achieving its maximum at 0 (like $h(y)=-\log y$) then $g$ is a continuous function achieving its maximum at $z$. Let $(u_n)_n$ be an increasing sequence of thresholds tending to infinity. So, $U_n=\{y:g(y)>u_n\}$ is a sequence of balls with diameter shrinking to zero. Now, we can rewrite Theorem \ref{Theorem:B_r} for the case $\Lambda=\{z\}$:
\begin{theorem}
      Let $(\Omega,T,\mu)$ be a $\phi$-mixing dynamical system where $T$ is measure preserving and $\mu$ is a right $\phi$-mixing for partition $\mathcal{A}$ of $\Omega$ where $\phi$ decays at least polynomially with power $m>1$.
    Let $g:\Omega\to \R\cup\{\infty\}$ where $g(y)=h(d(y,z))$. Assume that $0<r_n<r$ satisfies $\mu(B_{r+r_n}(z))\setminus \mu(B_{r-r_n}(z))=o(1)\mu(B_r(z))$. Let $\kappa_n$ be the smallest positive integer for which diam $(\mathcal{A}^{\kappa_n}) \le r_n$ and assume:
    \begin{enumerate}
    \item $\kappa_n\mu(U_n)^{c'}\to 0$ for some $c'\in(0,1)$;
    \item $U_n$ has small boundary: there exists $C>0$ and $m'>1$ so that $\mu(\bigcup_{B\in\mathcal{A}^j,B\cap B_{r_n}(\partial U_n)\ne \emptyset}B)\le Cj^{-m'}$ for all $n$ and $j\le \kappa_n$.
    \item for the case $z$ is periodic, $\theta=\lim_{n\to\infty}\frac{\mu(U_n\cap T^{-p}U_n)}{\mu(U_n)}$ exists.
\end{enumerate}
   Let $\alpha$ be a number in $[0,\infty)$ such that for $\alpha\in(0,1)$, we have $\alpha>\max\{\frac{1}{m+1},c'\}$. Then for all $s\in\R^+$,
\[  L_{\alpha,s} (z) = 
   \begin{cases} 
   1 & \text{if $z$ is non-periodic,} \\
   1-\theta & \text{if $z$ is periodic.} 
  \end{cases}
\]
\end{theorem}

\subsection{$\Lambda$ as a Cantor set}\label{Section: Cantor set}

Now we assume $\Lambda$ is a Cantor set. Here, for simplicity, we only work with the ternary Cantor set. 

Consider $\Omega=[0,1]$ and $T:\Omega\circlearrowleft$ defined by $T(x)=3x$ mod 1. Let $\mu$ be the Lebesgue measure on $[0,1]$. Then, $\mu$ is a probability measure and $T$ is $\mu$ preserving map. Assume $\mathcal{A}=\{[0,\frac{1}{3}),[\frac{1}{3},\frac{2}{3}),[\frac{2}{3},1]\}$, is our partition on $\Omega$. Lebesgue measure with respect to $\mathcal{A}$ is $\psi$-mixing, so by Remark \ref{remark:phi=psi}, Lebesgue measure is $\phi$-mixing that also decays exponentially. 

Now assume the decreasing sequence $(U_n)_n$ defined by $U_0=[0,1]$, and $U_{n+1}$ is obtained by removing the middle third of each connected component of $U_n$. Then we have $\bigcap_nU_n=\{\Lambda\}$. Note that $\mu(U_n)=(\frac{2}{3})^n$ and the outer j-cylinder approximation of $U_n$ is $U_j$ i.e. $U_n^j=U_j$, for $j\ge1$. Now we want to check that $(U_n)_n$ is a good neighbourhood system. We have,

\begin{enumerate}[(C1)]
    \item obviously $U_n\in\mathcal{A}^n$ so $\kappa_n=n$, and since $\mu(U_n)=(\frac{2}{3})^n$, by choosing any $c\in(0,1)$ we have $\kappa_n\mu(U_n)^c=n(\frac{2}{3})^{cn}\to 0$. So $(U_n)_n$ satisfies in (N\ref{Def:adapted1}), \label{cantor check N1}
    \item as we have $$\mu(U_n^j)\le\mu(U_n)+\mu(U_j)\le \mu(U_n)+4(j)^{-2}, \quad \forall j<n,$$ so $(U_n)_n$ satisfies in (N\ref{Def:adapted2}).
\end{enumerate}
Hence $(U_n)_n$ is a good neighbourhood system. Now we need to compute the extremal index. This has done as part of proof of \cite[Theorem 8.2]{HYang20}, and we simply state it in the following lemma.
\begin{lemma}
    \cite[Theorem 8.2]{HYang20}. Assume the $\psi$-mixing dynamical system $([0,1],T,\mu)$ where $T(x)=3x$ mod 1 and $\mu$ is the Lebesgue measure. Suppose $\Lambda$ is the Cantor ternary set and $(U_n)_n$ is the nested sequence defined above. Then $\theta_1=\frac{1}{3}$. 
\end{lemma}

 Now that all conditions of Theorem \ref{Theorem:Gibbs-Markov} holds for the Cantor set, we can restate it as follows:
 \begin{theorem}\label{Theorem: Cantor set}
     Let $\alpha\in(c,\infty)$ where $c$ is defined in (C\ref{cantor check N1}). For the uniformly expoanding map $T(x)=3x$ mod 1 on $[0,1]$, the Cantor ternary set $\Lambda$, the nested sets $(U_n)_n$, and any $s\in\R^+$, we have,
     $$L_{\alpha,s} (\Lambda)=\frac{1}{3}.$$
 \end{theorem}
\begin{remark}
    Theorem \ref{Theorem: Cantor set} can be stated more generally for the Cantor sets discussed in \cite{FFRSoares} with some modifications.
\end{remark}

\subsection{$\Lambda$ as a submanifold of Anosov maps}\label{Example: Cat map}
In this section, we aim to obtain the results for a submanifold of an Anosov map. An Anosov map is an example of uniformly hyperbolic dynamical system. It is a diffeomorphism of a compact manifold where the entire tangent space can be split into two invariance subspaces, one that is uniformly contracted and one is uniformly expanded under iteration of the map. The Arnold cat map on the torus, which is discussed in detail in \cite[Chapter 6.4]{HassKatok95}, is one of the classical and important examples of Anosov diffeomorphism.

Let $\Omega=\T^2=\frac{\R^2}{\Z^2}$ and 
\[ T=
\begin{bmatrix}
2 & 1 \\
1 & 1
\end{bmatrix}
\]
induces the cat map on the torus. The Cat map is uniformly hyperbolic, so admits a finite Markov partition $\mathcal{A}$. Suppose two-dimensional Lebesgue measure $\mu$ on $\T^2$, that is exponentially $\psi$-mixing dynamical system with respect to its partition (see \cite[Section 2.1]{CHNicol}). Hence $(\Omega, T, \mu)$ is an exponentially $\psi$-mixing dynamical system. Let $\Lambda\subset\T$ be a line segment with finite length $l(\Lambda)$ and the direction vector $\hat\Lambda$, which describes the direction of the displacement of the points after applying the map. We need an observable here, same as usual, achieve its maximum ($+\infty)$ on $\Lambda$.
Let $g(x)=-\log(d(x,\Lambda))$ be an observable from $\T^2$ to $\R$.
 
 Note that the Jacobian matrix of $T$, $DT$, has two unit eigenvectors $\nu^+$ and $\nu^-$ corresponding to the eigenvalues $\lambda_+=\lambda>1$ and $\lambda_-=\frac{1}{\lambda}<1$. If $\Lambda$ is aligend with the unstable direction, we lift $\Lambda$ to $\hat\Lambda$ on a fundamental domain of the cover $\R^2$ of $\T^2$ and write $\hat\Lambda=\hat{p_1}+t_1\nu^+$, $t_1\in [0,l(\Lambda]$, $\hat{p_1}\in\R^2$. Thus, $\Lambda=\pi(\hat{p_1}+t_1\nu^+)$, where $\pi:\R^2\to \T^2$ is the projection. We write the endpoint of $\hat\Lambda$ as $\hat p_2$, i.e. $\hat p_2=\hat{p_1}+l(\Lambda)\nu^+$. Similarly if $\Lambda$ is aligend with the stable direction, we lift $\Lambda$ to $\hat\Lambda$ on a fundamental domain of the cover $\R^2$ of $\T^2$ and write $\Lambda=\pi(\hat{p_1}+t_1\nu^-)$, and the endpoint of $\hat\Lambda$ as $\hat p_2=\hat{p_1}+l(\Lambda)\nu^-$.

Assume that $(u_n)_n$ is a non decreasing sequence of real numbers. Let $U_n=\{x:g(x)>u_n\}$. If we consider $\delta_n=e^{-u_n}$, then we have $$B_{\delta_n}(y)=\{x: d(x,y)<e^{-u_n}\}=\{x:\log d(x,y)>u_n\}=U_n.$$ The approximation of $U_n$ for $r_n=\delta_n^2=e^{-2u_n}\to 0$ as $n\to \infty$ are $U_n^-=U_n\setminus\overline{\{x: d(x,y)<e^{-u_n}\}}$ and $U_n^+=\bigcup_{y\in U_n} \{x: d(x,y)<e^{-u_n}\}$. Since $\mu$ is a Lebesgue measure, so it is straightforward to verify $$\mu(U_n^+\setminus U_n^-)=o(1)\mu(U_n).$$
Due to hyperbolicity of $T$, there exists $C>0$ such that diam $(\mathcal{A}^{n})<C\lambda^{-n}$. So, if we consider $\kappa_n=\lfloor\frac{\ln C+2 u_n}{\ln \lambda}\rfloor+1=\mathcal{O}(\log n)$, then we have
\begin{align*}
    diam (\mathcal{A}^{\kappa_n})\le diam (\mathcal{A}^{\frac{\ln C+2 u_n}{\ln \lambda}+1})&< C\lambda^{-\left(\frac{\ln C+2 u_n}{\ln \lambda}\right)-1}\\&=C \lambda^{\frac{\ln C^{-1}}{\ln \lambda}}\lambda^{\frac{-2u_n}{\ln \lambda}}\lambda^{-1}\\&=C\frac{1}{C}e^{-2u_n}\lambda^{-1}<e^{-2u_n}=r_n.
\end{align*}

Hence diam $(\mathcal{A}^{\kappa_n})<r_n$. As $\mu(U_n)=\mu(B_{\delta_n}(\Lambda))$, we have $$\mu(U_n)\lesssim e^{-u_n}l(\Lambda)=\mathcal{O}(\frac{1}{n}).$$

Since $ \mathcal{O}(\log n)\gg \mathcal{O}(\frac{1}{n})$ as $n\to \infty$, we have $\kappa_n\mu(U_n)^{c'}\to 0$ for any $c'\in(0,1).$

Now that the conditions of Theorem \ref{Theorem:B_r} are satisfied, according to Remark \ref{Remark: B_r for Gibbs-Markov}, we can state the following theorem.
\begin{theorem}
    Let $(\T^2,T,\mu)$ be the cat map and assume $\alpha\in[0,\infty)$. Then for all $s\in\R^+$,

    \begin{enumerate}
        \item if $\Lambda$ is not aligend with the stable direction or the unstable direction then $L_{\alpha,s} (\Lambda)=1$,
        \item if $\Lambda$ is aligend with the unstable direction but $\{\hat{p_1}+t_1\nu^+, t_1\in\R\}$ has no periodic points, then $L_{\alpha,s} (\Lambda)=1$,
        \item if $\Lambda$ is aligend with the stable direction but $\{\hat{p_1}+t_1\nu^-, t_1\in\R\}$ has no periodic points, then $L_{\alpha,s} (\Lambda)=1$,
        \item if $\Lambda$ is aligned with the stable or unstable direction and $\Lambda$ contains a periodic point with prim period $p$, then $L_{\alpha,s} (\Lambda)=1-\lambda^{-p}$,
        \item if $\Lambda$ is aligend with the unstable direction and $\Lambda$ has no periodic points but $\{\hat{p_1}+t_1\nu^+, t_1\in\R\}$ contains a periodic point of prime period $p$, then 
        \[  L_{\alpha,s} (\Lambda) = 
   \begin{cases} 
   1 & \text{if $\Lambda\cap T^{-p}\Lambda=\emptyset$,} \\
   1-\lambda^{-p}\frac{|\hat p_2|}{l(\Lambda)} & \text{if $\Lambda\cap T^{-p}\Lambda\ne\emptyset$.} 
  \end{cases}
\]
        \item if $\Lambda$ is aligend with the stable direction and $\Lambda$ has no periodic points but $\{\hat{p_1}+t_1\nu^-, t_1\in\R\}$ contains a periodic point of prime period $p$, then 
        \[  L_{\alpha,s} (\Lambda) = 
   \begin{cases} 
   1 & \text{if $\Lambda\cap T^{-p}\Lambda=\emptyset$,} \\
   1-\lambda^{-p}\frac{|\hat p_2|}{l(\Lambda)} & \text{if $\Lambda\cap T^{-p}\Lambda\ne\emptyset$.} 
  \end{cases}
\]
    \end{enumerate}
\end{theorem}

\begin{proof} 
    To prove 1, first we need to calculate $\theta$. This is done in the proof of \cite[Theorem 8.3]{HYang20} and is equal to one. Now based on Theorem \ref{Theorem:B_r}, Remark \ref{Remark: B_r for Gibbs-Markov} and Remark \ref{remark: thete=theta1 in B_r}, we have $L_{\alpha,s} (\Lambda)=1$.

    For proving case 2 to 6, we can use a similar approach to what is done for 1. That is, calculating $\theta$ using the proof of case 2 to 6 of \cite[Theorem 2.1]{CHNicol}, and then concluding with Theorem \ref{Theorem:B_r}, Remarks \ref{Remark: B_r for Gibbs-Markov} and Remark \ref{remark: thete=theta1 in B_r}. 
\end{proof}

\end{document}